\documentclass[12pt, reqno]{amsart}
\usepackage{amssymb,latexsym,amsmath,amsfonts}
\usepackage{latexsym}
\usepackage[mathscr]{eucal}
\usepackage{bm}

\newcommand{\C}{\mathbb{C}}

\newcommand{\D}{\mathbb{D}}

\newcommand{\E}{\mathbb{E}}

\newcommand{\lm}{\lambda}

\newcommand{\ph}{\varphi}

\newcommand{\q}{\quad}

\newcommand{\n}{\lVert}

\newcommand{\Lra}{\Leftrightarrow}

\def ~{\hspace{1mm}}

        \def\textmatrix#1&#2\\#3&#4\\{\bigl({#1 \atop #3}\ {#2 \atop #4}\bigr)}
        \def\dispmatrix#1&#2\\#3&#4\\{\left({#1 \atop #3}\ {#2 \atop #4}\right)}

        \newtheorem{defn}{Definition}[section] 

        \newtheorem{lem}[defn]{Lemma}
        \newtheorem{thm}[defn]{Theorem}

\begin{document}
\title[A generalized Schwarz lemma]
{A generalized Schwarz lemma for two domains related to
$\mu$-synthesis}

\author[Sourav Pal]{Sourav Pal}
\address[Sourav Pal]{Mathematics Department, Indian Institute of Technology Bombay,
Powai, Mumbai - 400076, India.} \email{sourav@math.iitb.ac.in}

\author[Samriddho Roy]{Samriddho Roy}
\address[Samriddho Roy]{Mathematics Department, Indian Institute of Technology Bombay,
Powai, Mumbai - 400076, India.} \email{sroy@math.iitb.ac.in
}



\thanks{The first author is supported by the Seed Grant of IIT Bombay, the CPDA and the INSPIRE
Faculty Award (Award No. DST/INSPIRE/04/2014/001462) of DST,
India. The second author is supported by a Ph.D fellowship from
the University Grand Commission of India.}

\begin{abstract}
We present a set of necessary and sufficient conditions that
provides a Schwarz lemma for the tetrablock $\mathbb E$. As an
application of this result, we obtain a Schwarz lemma for the
symmetrized bidisc $\mathbb G_2$. In either case, our results
generalize all previous results in this direction for $\mathbb E$
and $\mathbb G_2$.
\end{abstract}

\maketitle

\section{Introduction}

\noindent The aim of this article is to prove an explicit Schwarz
lemma for two polynomially convex domains related to
$\mu$-synthesis, the symmetrized bidisc $\mathbb G_2$ and the
tetrablock $\mathbb E$, defined as
\begin{align*}
& \mathbb G_2=\{(z_1+z_2,z_1z_2)\,:\, z_1,z_2 \in\mathbb D  \}\,,
\\ & \mathbb E = \{ (x_1,x_2,x_3)\in\mathbb C^3\,:\,
1-zx_1-wx_2+zwx_3\neq 0 \,, z,w \in\overline{\mathbb D} \}.
\end{align*}
The classical Schwarz lemma gives a necessary and sufficient
condition for the solvability of a two-point interpolation problem
for analytic functions from the open unit disc $\mathbb D$ to
itself, and describes the extremal functions. It has substantial
generalizations in which the two copies of $\mathbb D$ are
replaced by various other domains \cite{Din}, typically either
homogeneous or convex. In this paper, we prove a sharp Schwarz
lemma for two domains which are neither homogeneous nor convex. We
believe that our results will throw new lights on the spectral
Nevanlinna{Pick problem, which is to interpolate from the unit
disc to the set of $k \times k$ matrices of spectral radius no
greater than $1$ by analytic matrix functions.\\

We first produce several independent necessary and sufficient
conditions under which there exists an analytic function from
$\mathbb D$ to $\mathbb E$ that solves a two-point Nevanlinna-Pick
interpolation problem for the tetrablock. This, in a way,
generalizes the previous result of Abouhajar, White and Young in
this direction, \cite{awy}. We mention here that the domain
tetrablock was introduced by these three mathematicians in
\cite{awy} and that this domain has deep connection with the most
appealing and difficult problem of $\mu$-synthesis (see
\cite{bharali} for further details). Since then the tetrablock has
attracted considerable attention
\cite{tirtha,EKZ,EZ,sourav1,sourav2,bharali,sourav3,zwonek1}. So,
our first main result is the following Schwarz lemma for the
tetrablock.

\begin{thm} \label{schwarzL}
    Let $\lm_0 \in \D\setminus \{0\}$ and let $x=(a,b,p) \in \E$.  Then the following conditions are equivalent:
        \item[$(1)$] there exists an analytic function $\ph: \D \to \bar\E$ such that
        $\ph(0)=(0,0,0)$ and $\ph(\lm_0) = x$;
        \item[$(1^\prime)$] there exists an analytic function $\ph: \D \to \E$ such that
        $\ph(0)=(0,0,0)$ and $\ph(\lm_0) = x$;
        \item[$(2)$]
        \[
        \max \left\{ \frac{|a-\bar b p|+ |ab-p|}{1-|b|^2},
        \frac{|b-\bar a p|+|ab-p|}{1-|a|^2} \right\} \leq |\lm_0|;
        \]
        \item[$(3)$] either $|b| \leq |a|$ and
        \[
        \frac{|a-\bar b p|+ |ab-p|}{1-|b|^2} \leq |\lm_0|
        \]
        or $|a| \leq |b|$ and
        \[
        \frac{|b-\bar a p|+|ab-p|}{1-|a|^2} \leq |\lm_0|;
        \]
        \item[$(4)$] there exists a $2\times 2$ function $F$ in the Schur class such that
        \[
        F(0) = \left[ \begin{array}{cc} 0 & * \\ 0 & 0 \end{array} \right] \mbox{ and }
        F(\lm_0) = A = [a_{ij}]
        \]
        where $x = (a_{11},a_{22}, \det A).$\\

        \item[$(5)$] either $|b|\leq |a|$ and $\n \Psi(.,x) \n _{H^\infty}
        \leq
        |\lm_0|$\\

        or $|a|\leq |b|$ and $\n \Upsilon(.,x) \n _{H^\infty} \leq
        |\lm_0|;$\\

        \item[$(6)$] either $|b|\leq |a|$ and
        \[
        |a|^2 - |\lm_0|^2|b|^2 +|p|^2 + 2\Big| |\lm_0|^2b - \bar ap \Big| \leq |\lm_0|^2
        \]
        or $|a|\leq |b|$ and
        \[
        |b|^2 - |\lm_0|^2|a|^2 +|p|^2 + 2\Big| |\lm_0|^2a - \bar bp \Big| \leq |\lm_0|^2;
        \]

        \item[$(7)$] either $\lm_0 - az - b\lm_0w + pzw \neq 0$ for all $z,w \in \D$
        and $|b|\leq |a|$\\

        or $\lm_0 - a\lm_0z -bw + pzw \neq 0$ for all $z,w \in \D$ and
        $|a|\leq|b|$;\\

        \item[$(8)$] either $|b|\leq |a|$ and
        \[
        |\lm_0||a-\bar b p| + ||\lm_0|^2b-\bar a p|+
        |p|^2 \leq |\lm_0|^2
        \]
        or $|a|\leq |b|$ and
        \[
        |\lm_0||b-\bar a p| + ||\lm_0|^2a-\bar b p|+
        |p|^2 \leq |\lm_0|^2 ;
        \]

        \item[$(9)$] either $|b|\leq |a|$, $|p|\leq |\lm_0|$ and
        \[
        |a|^2+|\lm_0b|^2-|p|^2 + 2|\lm_0||ab-p| \leq |\lm_0|^2
        \]
        or $|a|\leq |b|$, $|p|\leq |\lm_0|$ and
        \[
        |b|^2+|\lm_0a|^2-|p|^2 + 2|\lm_0||ab-p| \leq |\lm_0|^2
        \]

        \item[$(10)$] either $|b|\leq |a|$, $|p| \leq |\lm_0|$ and there exist $\beta_1,\beta_2\in\mathbb
        C$ with $|\beta_1|+|\beta_2| \leq 1$ such that
        \[
        a=\beta_1\lm_0+\bar{\beta}_2 p \text{ and }
        b\lm_0=\beta_2 \lm_0+\bar{\beta}_1p
        \]
        or $|a|\leq |b|$, $|p| \leq |\lm_0|$ and there exist $\beta_1,\beta_2\in\mathbb
        C$ with $|\beta_1|+|\beta_2| \leq1$ such that
        \[
        a\lm_0=\beta_1\lm_0+\bar{\beta}_2 p \text{ and }
        b=\beta_2\lm_0+\bar{\beta}_1p.
        \]

\end{thm}

The proof of the theorem is given in Section 3. The functions
$\Psi$ and $\Upsilon$ play major role here. These two functions
were introduced in \cite{awy} and we shall briefly describe them
in Section 2. Being armed with the Schwarz lemma for the
tetrablock, we apply the result to the symmetrized bidisc to
obtain a generalized Schwarz lemma for $\mathbb G_2$, which is
another main result of this paper and is presented as Theorem
\ref{thm:Schwarz lemma} in Section 4. This is possible because of
the underlying relationship between the two domains $\mathbb E$
and $\mathbb G_2$ which has been described in Lemma
\ref{lem:tirtha} and Theorem \ref{thm:appl} in Section 4. This
result generalizes the Schwarz lemma obtained by Agler, Young in
\cite{ay-blms} and also independently by Nokrane and Ransford in
\cite{NR}.

\section{Background material}

\noindent We recall from \cite{awy} the following two rational
functions $\Psi,\Upsilon$ which play central role in the study of
complex geometry of $\E$.

\begin{defn} \label {defPsi}
    For $z \in \C$ and $x=(x_1,x_2,x_3) \in \C^3$ the functions
    $\Psi$ and $\Upsilon$ are defined as
    \begin{eqnarray}
    \Psi(z,x) &=& \frac{x_3 z - x_1}{x_2z - 1}, \label{dPsi}\\
    \Upsilon(z,x) &=& \Psi(z,(x_2,x_1,x_3)) = \frac{x_3z-x_2}{x_1z-1}. \label{dUps}
    \end{eqnarray}
\end{defn}

Also we define

\begin{equation}\label{dB}
D(x) = \sup_{z \in \D} | \Psi(z,x)| = || \Psi(.,x)||_{H^\infty}.
\end{equation}

We interpret $\Psi(.,x)$ to be the constant function equal to
$x_1$ in the event that $x_1 x_2 = x_3$; thus $\Psi(z,x)$ is
defined when $zx_2 \ne 1$ or $x_1 x_2 = x_3$. The quantity $D(x)$
is finite (and $\Psi(.,x) \in H^\infty$) if and only if either
$x_2 \in \D$ or $x_1x_2=x_3$. Indeed, for $x_2 \in \D$, the linear
fractional function $\Psi(.,x)$ maps $\D$ to the open disc with
centre and radius
\begin{equation}  \label{imPsi}
\frac{x_1-\bar x_2 x_3}{1-|x_2|^2}, \quad  \frac{|x_1 x_2 - x_3|}{1-|x_2|^2}
\end{equation}
respectively.  Hence
\begin{equation}  \label{formD}
D(x) = \left \{ \begin{array}{ll} \displaystyle \frac {|x_1 - \bar x_2 x_3| + |x_1 x_2 - x_3|}
{1-|x_2|^2} & \mbox{ if $|x_2| < 1$} \\
|x_1| & \mbox{ if $x_1x_2=x_3$}\\
\infty & \mbox { otherwise.} \end{array}
\right .
\end{equation}
Similarly, if $x_1 \in \D$, $\Upsilon(.,x)$ maps $\D$ to the open disc with
centre and radius
\[
\frac{x_2 - \bar x_1 x_3}{1-|x_1|^2}, \quad  \frac{|x_1 x_2 - x_3|}{1-|x_1|^2}
\]
respectively. It was shown in \cite{awy} that the closure
$\overline{\E}$ of the tetrablock is the following set

\begin{equation}\label{sc:eq1}
\mathbb E = \{ (x_1,x_2,x_3)\in\mathbb C^3\,:\,
1-zx_1-wx_2+zwx_3\neq 0 \textup{ whenever } |z|< 1, |w|< 1 \}.
\end{equation}
 In \cite{awy}, the points in the sets $\E$ and its closure
 $\overline{\E}$ are characterized in the following way.

\begin{thm}[Abouhajar, White, Young] \label{AWY}

For $x = (x_1, x_2, x_3)$ in $\mathbb C ^3$, the following are
equivalent.

\begin{enumerate}

\item $ x\in\mathbb E$ (respectively $x\in\overline{\mathbb E}$),

\item $\| \Psi ( \cdot , x ) \|_{H^\infty} < 1$ (respectively
$\leq 1$),

\hspace*{-0.5in} (2') $\| \Upsilon ( \cdot , x ) \|_{H^\infty} <
1$ (respectively $\leq 1$),

\item $| x_1 - \bar{x}_2 x_3 | + |x_1 x_2 - x_3 | < 1 - |x_2|^2$
(respectively $\leq 1 - |x_2|^2$),

\hspace*{-0.5in} (3') $| x_2 - \bar{x}_1 x_3 | + |x_1 x_2 - x_3 |
< 1 - |x_1|^2$ (respectively $\leq 1 - |x_1|^2$),

\item $|x_1|^2 - |x_2|^2 + |x_3|^2 + 2 |x_2 - \bar{x}_1 x_3
    | < 1$ (respectively $\leq 1$) and $|x_2| < 1$
    (respectively $\leq 1$),

\hspace*{-0.5in} (4') $ - |x_1|^2 + |x_2|^2 + |x_3|^2 + 2 |x_1 -
\bar{x}_2 x_3 | < 1$ (respectively $\leq 1$) and $|x_1| < 1$
(respectively $\leq 1$),

\item $|x_1|^2 + |x_2|^2 - |x_3|^2 + 2 |x_1x_2 - x_3 | < 1$
    (respectively $\leq 1$) and $|x_3| < 1$ (respectively
    $\leq 1$),

\item $ | x_1 - \bar{x}_2 x_3 | + |x_2 - \bar{x}_1 x_3 | < 1 -
|x_3|^2$ (respectively $\leq 1 - |x_3|^2$),

\item there exists a matrix $A=(a_{i,j}) \in \mathbb M _2$ with
$\| A \| < 1$ (respectively $\leq 1$) such that
$x=(a_{11},a_{22},\det A)$,

\item there exists a symmetric matrix $A=(a_{i,j}) \in \mathbb M
_2$ with $\| A \| < 1$ (respectively $\leq 1$) such that
$x=(a_{11},a_{22},\det A)$,

\item $| x_3 | < 1$ (respectively $\leq 1$) and there are complex
numbers $\beta_1$ and $\beta_2$ with $| \beta_1 | + |\beta_2 | <
1$ (respectively $\leq 1$) such that
    \[
    x_1 = \beta_1 + \overline{\beta}_2 x_3 \mbox{ and }  x_2 = \beta_2 + \overline{\beta}_1 x_3.
    \]

\end{enumerate}

\end{thm}

\section{Proof of Theorem \ref{schwarzL}}

\noindent The equivalence of $(1)-(4)$ was established in Theorem
$1.2$ of \cite{awy}. We shall prove here the rest parts.\\

\noindent $(3) \Lra (5).$ Condition $(3)$ can be written as
\begin{align*}
\text{either} \q & \frac{|a-\bar b p|+ |ab-p|}{1-|b|^2} \leq
|\lm_0| \\ \text{or } \q & \frac{|b-\bar a p|+|ab-p|}{1-|a|^2}
\leq |\lm_0|.
\end{align*}
Since
\begin{align*}
&\n \Psi(.,x) \n _{H^\infty} = \frac{|a-\bar b p|+ |ab-p|}{1-|b|^2}  \q \text{and} \\
& \n \Upsilon(.,x) \n _{H^\infty} = \frac{|b-\bar a
p|+|ab-p|}{1-|a|^2}
\end{align*}

we have that condition $(3)$ is equivalent to condition $(5)$.\\

\noindent $(5) \Lra (6).$ From the definition of $\Psi(.,x)$ (see
(\ref{dPsi})), it is evident by an application of the Maximum
Modulus Principle that $ \n \Psi(.,x) \n _{H^\infty} \leq |\lm_0|$
holds if and only if
$$ \dfrac{|pz - a|}{|bz - 1|} \leq |\lm_0|  \q \textrm{for all } z\in\mathbb{T}.$$
Now for all $z\in\mathbb T$,
\begin{align*}
& \dfrac{|pz - a|}{|bz - 1|} \leq |\lm_0|\\
& \Lra |pz - a|^2 \leq |\lm_0|^2 |bz - 1|^2\\
&\Lra |pz|^2 + |a|^2 - 2 \operatorname{Re}(\bar apz) \leq |\lm_0|^2 \Big(|bz|^2 + 1 - 2 \operatorname{Re}(bz)\Big)\\
& \Lra |a|^2 - |\lm_0|^2|b|^2 + |p|^2 - |\lm_0|^2 + 2 \operatorname{Re}\big(z(|\lm_0|^2b - \bar ap)\big) \leq 0\\
&\Lra |a|^2 - |\lm_0|^2|b|^2 + |p|^2 - |\lm_0|^2 + 2\Big|
|\lm_0|^2b - \bar ap \Big | \leq 0.
\end{align*}

We obtained the last step by using the fact that $ |x| < k \Lra
\operatorname{Re}(zx)< k$, for all $z \in \mathbb{T}$.\\

Similarly we obtain
\begin{align*}
& \n \Upsilon(.,x) \n _{H^\infty} \leq |\lm_0| \\
& \Lra \dfrac{|pz - b|}{|az - 1|} \leq |\lm_0|  \q \textrm{for all } z\in\mathbb{T} \\
& \Lra |b|^2 - |\lm_0|^2|a|^2 + |p|^2 - |\lm_0|^2 + 2\Big|
|\lm_0|^2a - \bar bp \Big | \leq 0.
\end{align*}

So, $(5)\Lra (6)$ is evident now.\\

\noindent $(5) \Lra (7).$ It is evident from part-(2) of Theorem
\ref{AWY} that

\begin{align*}
& \n \Psi(.,x) \n _{H^\infty} \leq |\lm_0| \\
\Lra \q & \dfrac{|pz - a|}{|bz - 1|} \leq |\lm_0|  \q \textrm{for
all } z\in\mathbb{T} \\
\Lra \q  & \dfrac{|zp/\lm_0 - a/\lm_0|}{|bz - 1|} \leq 1  \q
\textrm{for all } z\in\mathbb{T} \\
\Lra \q & (a/\lm_0,b,p/\lm_0)\in \overline{\mathbb E} \\
\Lra \q & \lm_0 - az - b\lm_0w + pzw \neq 0 \q \text{for all }
z,w\in \D.
\end{align*}

Similarly from part-($2'$) of Theorem \ref{AWY}, we can conclude
that
\[
\n \Upsilon(.,x) \n _{H^\infty} \leq |\lm_0| \Lra \lm_0 - a\lm_0z
-bw + pzw \neq 0 \q \text{ for all }  z,w \in \D.
\]

\noindent $(7)\Lra (8).$ It is evident from the equation
(\ref{sc:eq1}) in Section 2 that $(7)$ holds if and only if
$(a/\lm_0,b,p/\lm_0)\in\overline{\E}$ or
$(a,b/\lm_0,p/\lm_0)\in\overline{\E}$. Part-(6) of Theorem
\ref{AWY} tells us that
\[
(a/\lm_0,b,p/\lm_0)\in\overline{\E} \Lra |\lm_0||a-\bar b p| +
||\lm_0|^2b-\bar a p|+|p|^2 \leq |\lm_0|^2.
\]

Also by part-(6) of Theorem \ref{AWY},
\[
(a,b/\lm_0,p/\lm_0)\in\overline{\E} \Lra |\lm_0||b-\bar a p| +
||\lm_0|^2a-\bar b p|+|p|^2 \leq |\lm_0|^2.
\]
Therefore, $(7)\Lra (8)$ holds.\\

\noindent $(7)\Lra (9).$ We apply part-(5) of Theorem \ref{AWY} to
get

\begin{align*}
\q &\lm_0 - az - b\lm_0w + pzw \neq 0 \,, \q \text{ for all } z,w \in \D \\
\Lra \q &(a/\lm_0,b,p/\lm_0)\in\overline{\E}\\
\Lra \q & |a|^2+|\lm_0b|^2-|p|^2 + 2|\lm_0||ab-p| \leq |\lm_0|^2.
\end{align*}

Similarly one obtains
\begin{align*}
\q & \lm_0 - a\lm_0z -bw + pzw \neq 0 \,, \q \text{ for all } z,w
\in \D \\
\Lra \q &(a,b/\lm_0,p/\lm_0)\in\overline{\E}\\
\Lra \q & |b|^2+|\lm_0a|^2-|p|^2 + 2|\lm_0||ab-p| \leq |\lm_0|^2.
\end{align*}
Therefore, $(7)$ is equivalent to $(9)$.\\

\noindent $(7)\Lra (10).$ We again mention that $(7)$ holds if and
only if $(a/\lm_0,b,p/\lm_0)\in\overline{\E}$ or
$(a,b/\lm_0,p/\lm_0)\in\overline{\E}$. So, by part-(9) of Theorem
\ref{AWY}, $7$ holds if and only if either $|p/\lm_0| \leq 1$ and
there exist $\beta_1,\beta_2\in\mathbb
        C$ with $|\beta_1|+|\beta_2| \leq 1$ such that
        \[
        a/\lm_0=\beta_1+\bar{\beta}_2(p/\lm_0) \text{ and } b=
        {\beta}_2+\bar{\beta}_1(p/\lm_0)
        \]
        which is same as saying that $|p|\leq |\lm_0|$ and
        \[
        a=\beta_1\lm_0+\bar{\beta}_2 p \text{ and }
        b\lm_0=\beta_2 \lm_0+\bar{\beta}_1p\,;
        \]
        or, $|p/\lm_0| \leq 1$ and there exist $\beta_1,\beta_2\in\mathbb
        C$ with $|\beta_1|+|\beta_2| \leq1$ such that
        \[
        a=\beta_1+\bar{\beta}_2(p/\lm_0) \text{ and }
        b/\lm_0=\beta_2+\bar{\beta}_1(p/\lm_0)
        \]
        which is equivalent to saying that $|p|\leq|\lm_0|$ and
        \[
        a\lm_0=\beta_1\lm_0+\bar{\beta}_2 p \text{ and }
        b=\beta_2\lm_0+\bar{\beta}_1p.
        \]

\section{Application to the symmetrized bidisc}

\noindent The symmetrized bidisc $\mathbb G_2$ was introduced by a
group of control theorists and later have been extensively studied
by the complex analysists, geometers and operator theorists during
past three decades. We list here a few of the numerous references
about the symmetrized bidisc for the readers, \cite{ay-blms,
ay-jot, tirtha-sourav, tirtha-sourav1, sourav, sourav1,
pal-shalit}. This is another domain which has root in the
$\mu$-synthesis problem. The following result provides a beautiful
connection between the two domains $\mathbb E$ and $\mathbb G_2$.

\begin{lem}[Lemma 3.2, \cite{tirtha}]\label{lem:tirtha}
Let $(x_1,x_2,x_3)\in\mathbb C^3$. Then $(x_1,x_2,x_3)\in\mathbb
E$ if and only if $(x_1+zx_2,zx_3)\in\mathbb G_2$ for all
$z\in\mathbb T$.
\end{lem}

Among the various characterizations of the points of $\mathbb G_2$
(see Theorem 1.1, \cite{ay-jot}), the following is very elegant.

\begin{thm}[Theorem 1.1, \cite{ay-jot}]\label{thm:AY}
Let $(s,p)\in\mathbb C^2$. Then $(s,p)\in\mathbb G_2$ if and only
if $|p|<1$ and there exists $\beta \in\mathbb C$ with $|\beta|<1$
such that
\[
s=\beta+\bar{\beta}p.
\]

\end{thm}

The following result will play central role in determining the
interpolating function in the Schwarz lemma for the symmetrized
bidisc.

\begin{thm}\label{thm:appl}
The symmetrized bidisc $\mathbb G_2$ can be analytically embedded
inside the tetrablock via the follwoing map
\begin{align*}
f : &\; \mathbb G_2 \rightarrow \mathbb E \\& (s,p) \mapsto
(\frac{s}{2},\frac{s}{2},p)\,.
\end{align*}
Conversely, the map
\begin{align*}
g: & \mathbb E \rightarrow \mathbb G_2 \\& (a,b,p) \mapsto (a+b,p)
\end{align*}
maps the tetrablock analytically onto the symmetrized bidisc and
when restricted to $f(\mathbb G_2)$, becomes the inverse of the
map $f$. Moreover, both $f$ and $g$ map the origin to the origin.
\end{thm}

\begin{proof}
The map
\begin{align*}
f : &\; \mathbb G_2 \rightarrow \mathbb E \\& (s,p) \mapsto
(\frac{s}{2},\frac{s}{2},p)\,,
\end{align*}
is clearly a one-one and holomorphic map. All we need to show is
that it maps $\mathbb G_2$ into the tetrablock. Let
$(s,p)\in\mathbb G_2$. Then by Theorem \ref{thm:AY}, $s=\beta
+\bar{\beta}p$, where
\[
\beta = \frac{s-\bar s p}{1-|p|^2} \text{ and } |\beta| <1\,.
\]
Now
\[
\frac{s}{2}=\frac{\beta}{2} +\overline{\frac{\beta}{2}}p
\]
together with $\left| \dfrac{\beta}{2} \right|+ \left|
\dfrac{\beta}{2} \right|<1$ imply by part-(9) of Theorem \ref{AWY}
that
$\left( \dfrac{s}{2},\dfrac{s}{2},p \right)\in \mathbb E$. Hence we are done.\\

Conversely, the function $g$ maps $\mathbb E$ to $\mathbb G_2$ by
Lemma \ref{lem:tirtha} and is evidently holomorphic. Now
\[
f(G)=\{ (a,b,p)\in\mathbb E\,:\, a=b  \}.
\]
Because, if $(a,a,p)\in \mathbb E$, then by Lemma
\ref{lem:tirtha}, $(2a,p)\in\mathbb G_2$ and by the previous part
$(a,a,p)=f(2a,p)$. This also proves that $g|_{f(\mathbb G_2)}$ is
the inverse of $f$. Needless to prove that both $f$ and $g$  map the origin to the origin.

\end{proof}

\subsection{A generalized Schwarz lemma for the symmetrized bidisc}

In this subsection, we apply Theorem \ref{schwarzL} to obtain a
Schwarz lemma for the symmetrized bidisc. This result generalizes
the previous results of Agler, Young and Nokrane and Ransford
\cite{ay-blms, NR}. We transfer the results of the tetrablock to
the symmetrized bidisc via the maps $f$ and $g$ and obtain the
following (generalized) Schwarz lemma for the symmetrized bidisc.

\begin{thm}\label{thm:Schwarz lemma}
Let $\lambda_0 \in \mathbb D \setminus \{ 0 \}$ and let $(s,p)\in
\mathbb G_2$. Then the following are equivalent:

        \item[$(1)$] there exists an analytic function $\psi: \D \to
        \mathbb G_2$ such that
        $\ph(0)=(0,0)$ and $\psi(\lm_0) = (s,p)$;
        \item[$(1^\prime)$] there exists an analytic function $\psi: \D \to
        \Gamma_2$ such that
        $\ph(0)=(0,0)$ and $\psi(\lm_0) = (s,p)$;
        \item[$(2)$]
        \[
        \frac{2|s-\bar s p|+ |s^2-4p|}{4-|s|^2} \leq |\lm_0|;
        \]

        \item[$(3)$] there exists a $2\times 2$ function $F$ in the Schur class such that
        \[
        F(0) = \left[ \begin{array}{cc} 0 & * \\ 0 & 0 \end{array} \right] \mbox{ and }
        F(\lm_0) = A = [a_{ij}]\,,
        \]
        where $a_{11}=a_{22}=\dfrac{s}{2}$ and $p=\det A$.\\

        \item[$(4)$] $\n \Psi(.,(\dfrac{s}{2},\dfrac{s}{2},p)) \n _{H^\infty}
        \leq |\lm_0|$\\

        \item[$(5)$]
        $
        (1-|\lm_0|^2)|s|^2 +4|p|^2 + 4\Big| |\lm_0|^2s - \bar s p \Big| \leq
        4|\lm_0|^2\,;
        $ \\

        \item[$(6)$] $2\lm_0 - (z +\lm_0 w)s+ 2pzw \neq 0$ for all $z,w \in
        \D$;\\

        \item[$(7)$]
        $
        |\lm_0||s-\bar s p| + ||\lm_0|^2s-\bar s p|+
        2|p|^2 \leq 2|\lm_0|^2\,;
        $ \\

        \item[$(8)$] $|p|\leq |\lm_0|$ and
        \[
        (1+|\lm_0|^2)|s|^2-4|p|^2 + 2|\lm_0||s^2-4p| \leq
        4|\lm_0|^2\,;
        \]

        \item[$(9)$] $|p| \leq |\lm_0|$ and there exist $\beta\in\mathbb
        C$ with $|\beta| \leq 1$ such that
        \[
        \text{either } {s}=2\beta\lm_0+2\bar{\beta} p\,
        \text{ or } \lm_0 s=2\beta \lm_0+2\bar{\beta}p\,.
        \]

\end{thm}

\begin{proof}
The proof follows easily from Theorem \ref{thm:appl} by applying
the functions $f$ and $g$. For $(s,p)\in\mathbb G_2$, we have
$f((s,p))=(\dfrac{s}{2}, \dfrac{s}{2},p)\in \mathbb E$ and Theorem
\ref{schwarzL} guarantees the existence of an analytic function
$\phi$ from $\mathbb D$ to $\mathbb E$ or to $\overline{\mathbb
E}$ that maps $0$ to $0$ and $\lm_0$ to $(\dfrac{s}{2},
\dfrac{s}{2},p)$, when the conditions $(2)-(10)$ are satisfied
with $a=b=\dfrac{s}{2}$. The conditions $(2)-(10)$ of Theorem
\ref{schwarzL} when $a=b=\dfrac{s}{2}$, have been represented as
conditions $(2)-(9)$ in this theorem. Hence under the conditions
$(2)-(9)$, we obtain the analytic function $\psi=g\circ \phi$ that
maps $0$ to $0$ and $\lm_0$ to $(s,p)$.\\

The converse is also true. That is, if there is a function $\phi$
from $\mathbb D$ to $\mathbb G_2$ such that $\phi (0)=0$ and
$\phi(\lm_0)=(s,p)$, then $f\circ \phi \,:\mathbb D \rightarrow
\mathbb E$ maps $0$ to $0$ and $\lm_0$ to $(\dfrac{s}{2},
\dfrac{s}{2},p)$. So, the conditions $(2)-(10)$ of Theorem
\ref{schwarzL} are satisfied with $a=b=\dfrac{s}{2}$, which is
same as saying that the conditions $(2)-(9)$ hold. Hence the proof
is complete.

\end{proof}

\end{document}